\documentclass[11pt,twoside]{article}
\usepackage{acta-m,euler}
\newtheorem{theo}{Theorem} 
\newtheorem{rem}{Remark} 
\newtheorem{lem}{Lemma} 
\newtheorem{cor}{Corollary} 

\newcommand{\R}{\mathbb{R}}

\renewcommand{\d}{{\ \rm d}}
\newcommand{\ep}{{\epsilon}}

\title{Asymptotic behavior of the solution of quasilinear parametric variational inequalities in a beam with a thin neck
}

\newcommand{\vol}{\textbf{2}, 1 (2010) 5--24}   
\newcommand{\amss}{\amssubj{35B40, 35B45
}}   
\newcommand{\keyw}{\keywords{asymptotic behavior, quasilinear parametric variational inequality
}}

\begin{document}
\maketitle

\oneauthor{Iuliana Marchis}
{Babe\c{s}-Bolyai University\\ Faculty of Psychology and Educational Sciences\\ Str. Kogalniceanu, Nr. 4\\ 400084 Cluj-Napoca, Romania
}
{krokusz@gmail.com}

\short{I. Marchis}{Asymptotic behavior of the solution of
parametric variational inequalities }

\begin{abstract}
In this paper we study the asymptotic behavior of the solution of quasilinear parametric variational inequalities posed in a cylinder with a thin neck, and we obtain the limit problem.
\end{abstract}

\section{Introduction}

The aim of the paper is to study the asymptotic behavior of the solution of quasilinear variational inequalities in a beam with a thin neck. Mathematically, this  notched beam is given by
\begin{equation}\label{hengerek}\nonumber
\Omega_\ep =\{(x_1, x')\in \R^3: -1 < x_1 < 1, |x'|<\ep \ \mbox{ if } |x_1|>t_\ep, |x'|<\ep r_\ep  \ \mbox{ if } |x_1|\leq t_\ep\},
\end{equation}
where $\ep$, $r_\ep$, and $t_\ep$ are positive parameters such that $\frac{\ep r_\ep}{t_\ep}\to 0$.

Previous work on domains of this type was done by Hale \& Vegas \cite{hale}, Jimbo \cite{jimbo1,jimbo2}, Cabib, Freddi, Morassi, \& Percivale \cite{cabib}, Rubinstein, Schatzman \& Sternberg \cite{rubi}, Casado-D\'iaz, Luna-Laynez \& Murat \cite{casado1, casado2} and Kohn \& Slastikov \cite{kohn}.

The most recent results are of Casado-D\'iaz, Luna-Laynez \& Murat \cite{casado2}. They studied the asymptotic behavior of the solution of a diffusion equation in the notched beam $\Omega_\ep$ and obtained at the limit a one-dimensional model.

In the present article the geometrical setting is the same as in \cite{casado2}, but we consider quasilinear variational inequalities instead of linear variational equalities.

The paper is organized as follows. In Section 2 the geometrical setting is described, the studied problem is given, and the assumptions for our results are formulated. In Section 3 the asymptotic behavior of the solution is studied. Some results from \cite{marchis} are recalled which, unfortunately, don't provide information about what happening near to the notch. Thus we need to prove some auxiliary results.   In Section 4 the limit problem is obtained. To prove the results in this section, we combine the ideas from \cite{mossino} with the adaptation to variational inequalities of the method used in \cite{casado2}.

\section{Setting the problem}

Let $\ep>0$ be a parameter, $r_\ep$ ($r_\ep>0$) and $t_\ep$ ($t_\ep>0$) be two sequences of real numbers, with
$$r_\ep \to 0, \ \ \ \ \ t_\ep \to 0, \ \ \mbox{ when } \ep\to 0.$$
We assume that
$$\frac{t_\ep}{r^2_\ep}\to \mu, \ \ \frac{\ep}{r_\ep}\to \nu, \ \ \ \mbox{ with } 0\leq \mu < +\infty, \ 0\leq \nu < +\infty,  \ \ \mbox{ when } \ep\to 0.$$
Let $S\subset \R^2$ be a bounded domain such that $0\in S$, which is sufficiently smooth to apply the Poincar\'e-Wirtinger inequality.

Define the following subsets of $\R^3$:
$$\Omega_\ep^-=(-1, -t_\ep)\times (\ep S), \ \ \Omega_\ep^0=[-t_\ep, t_\ep]\times (\ep r_\ep S),\ \ \Omega_\ep^+=(t_\ep,1)\times (\ep S),$$
$$\Omega_\ep = \Omega_\ep^- \cup \Omega_\ep^0 \cup \Omega_\ep^+, \ \mbox{ and } \  \Omega_\ep = \Omega_\ep^-  \cup \Omega_\ep^+.$$

\noindent $\Omega_\ep$ is a notched beam, the main part of the beam is $\Omega_\ep^1$ and the notched part $\Omega_\ep^0$.
A point of $\Omega^\ep$ is denoted by $x=(x_1, x')=(x_1,x_2,x_3)$.



Denote by
$$\Gamma_\ep^-=\{-1\}\times (\ep S) \ \mbox{ and }  \ \Gamma_\ep^+=\{1\}\times (\ep S)$$
the two bases of the beam, and let
$$\Gamma_\ep = \Gamma_\ep^- \cup \Gamma_\ep^+$$
be  the union of the two bases.


Denote
$$\mathcal{V}_\ep=\{V\in H^{1}(\Omega_\ep),\ \ V=0 \mbox{ on } \Gamma_\ep \}.$$



We consider the following problem:\\
Find $U_\epsilon \in M_\ep$ such that, for all $V_\epsilon\in M_\ep$,
\begin{align}\label{nagyegyen}
&\int_{\Omega_\epsilon}\left[A_\epsilon \Phi_\epsilon(x,U_\epsilon,B_\epsilon)\nabla U_\epsilon,\nabla(V_\epsilon-U_\epsilon) \right] \d x \geq 0
\end{align}
with $A_\epsilon$, $B_\epsilon$, and $\Phi_\epsilon$,  given functions,  $M_\ep$ a closed, convex, nonempty cone in $\mathcal{V}_\ep$.


This problem has applications in Physics. Bruno \cite{bruno} observed that when a ferromagnet has a thin neck, this will be preferred location for the domain wall. He also noticed that if the geometry of the neck varies rapidly enough, it can influence and even  dominate the structure of the wall.

Consider problem (\ref{nagyegyen}). We impose the following assumptions:

{\vskip 0.3cm}

{\bf (A1)} The matrix $A_\ep$ has the following form
\begin{equation}\nonumber
A_\ep(x) = \chi_{\Omega_\ep^1}(x) A^1 \left(x_1, \frac{x'
}{\ep} \right)+\chi_{\Omega_\ep^0}(x) A^0\left(\frac{x_1}{t_\ep},\frac{x'}{\ep r_\ep}\right),
\end{equation}
where  $A^1, A^0 \in L^\infty((-1,1)\times S)^{3\times 3}$.

{\vskip 0.3cm}

{\bf (A2)} The matrix $B_\ep$ has the following form
\begin{equation}\nonumber
B_\ep(x) = \chi_{\Omega_\ep^1}(x) B^1 \left(x_1, \frac{x'
}{\ep} \right)+\chi_{\Omega_\ep^0}(x) B^0\left(\frac{x_1}{t_\ep},\frac{x'}{\ep r_\ep}\right),
\end{equation}
where $B^1, B^0\in L^\infty((-1,1)\times S)^{3\times 3}$.

{\vskip 0.3cm}

{\bf (A3)} The functions $\Phi_\epsilon:\Omega_\ep\times\R \to \R^{3\times 3}$ and $\Psi_\epsilon:\Omega_\ep \times \R \to \R^3$ are Carath\'eodory mappings having the following form:
\begin{align*}
\Phi_\ep(x, \eta) &= \chi_{\Omega_\ep^1}(x) \Phi^1_\ep \left(x_1, \frac{x'
}{\ep}, \eta \right)+\chi_{\Omega_\ep^0}(x) \Phi^0_\ep \left(\frac{x_1}{t_\ep},\frac{x'}{\ep r_\ep}, \eta\right);
\end{align*}
for a.e. $x \in \Omega_\ep$, for all $\eta\in \R$; \\
for all $U_\ep \in L^2(\Omega_\ep)$, $W_\ep \in L^2(\Omega_\ep)^3$, $\Phi_\ep^1(\cdot, U_\ep(\cdot))W_\ep(\cdot), \Phi_\ep^0(\cdot, U_\ep(\cdot))W_\ep(\cdot) \in L^{2}((-1,1)\times S)^{3}$.
{\vskip 0.3cm}

\newpage

{\bf (A4)}
 {\it Coercivity condition}

\noindent There exist $C_1, C_2 > 0$ and $k_1\in L^\infty (\Omega_\ep)$ such that for all $\xi \in \R^3$, $\eta \in \R$
\begin{equation}\label{coer1}
[A_\ep(x) \Phi_\ep(x,\eta) B_\ep(x)\xi,\xi] \geq C_1 \|\xi\|^2+C_2 |\eta|^{q_1}-k_1(x) \ \ \mbox{a.e. } x\in \Omega_\ep
\end{equation}
for some $1<q_1<2$, for each $\ep>0$.
{\vskip 0.3cm}

{\bf (A5)} {\it Growth condition}

\noindent There exist $C>0$ and $\alpha \in L^{\infty}(\Omega_\ep)$ such that for all $\xi \in \R^3$, $\eta\in\R$
\begin{equation}\label{grow1}
\|A_\ep(x) \Phi_\ep(x, \eta)\xi \|\leq C\|\xi\|+C|\eta|+\alpha(x) \ \ \mbox{ a.e. } x\in \Omega_\ep,
\end{equation}
for each $\ep>0$.


{\vskip 0.3cm}

{\bf (A6)} {\it Monotonicity condition}

\noindent
For all $\xi, \tau \in \R^n$, $\eta \in \R$,
\begin{equation}\nonumber
\left[A_\epsilon(x)\Phi_\epsilon(x, \eta) B_\epsilon(x)\xi - A_\epsilon(x)\Phi_\epsilon(x, \eta) B_\epsilon(x)\tau, \xi-\tau \right]\geq 0, \ \mbox{ a. e. } x\in \Omega_\ep,
\end{equation}
for each $\ep>0$.

{\vskip 0.3cm}

{\bf (A7)} If $u_\ep \to u$ and $w_\ep \rightharpoonup w$ in $L^2(Y^1)$, then
$$\Phi^1_\ep(\cdot, u_\ep(\cdot))w(\cdot) \to \Phi^1(\cdot, u(\cdot))w(\cdot) \ \mbox{ strongly in } L^2(Y^1).$$

If $u_\ep \to u$ and $w_\ep \rightharpoonup w$ in $L^2(Z)$, then
$$\Phi^0_\ep(\cdot, u_\ep(\cdot))w(\cdot) \to \Phi^0(\cdot, u(\cdot))w(\cdot) \ \mbox{ strongly in } L^2(Z).$$









\section{Asymptotic behavior of the solution}

To study the asymptotic behavior we use the change of variables $y=y_\ep(x)$ given by
\begin{equation}\label{ch1}
y_1 = x_1 \ \ \ y'=\frac{x'}{\ep}
\end{equation}
 which transforms the beam (except the notch) in a cylinder of fixed diameter. This change of variable is classical in the study  of asymptotic behavior of variational equalities in thin cylinders or beams (see \cite{gustafsson}, \cite{muratsili}, \cite{sili}). We denote by $Y_\ep^-$, $Y_\ep^0$, $Y_\ep^+$, $Y_\ep$, and $Y_\ep^S$ the images of $\Omega_\ep^-$, $\Omega_\ep^0$, $\Omega_\ep^+$, $\Omega_\ep$, and $\Omega_\ep^S$ by the change of variables $y=y_\ep(x)$, i.e.
 $$Y_\ep^-=(-1, -t_\ep)\times S, \ \ Y_\ep^0=[-t_\ep, t_\ep]\times (r_\ep S), \ \ Y_\ep^+=(t_\ep,1)\times S,$$
 $$Y_\ep=Y_\ep^- \cup Y_\ep^0 \cup Y_\ep^+, \ \   Y^1_\ep = Y^-_\ep \cup Y_\ep^+.$$
 Denote by $Y^-$, $Y^+$, and $Y^1$ the ''limits''of $Y_\ep^-$, $Y_\ep^+$, and $Y_\ep^1$, i.e.
 $$Y^-=(-1,0)\times S, \ \ Y^+=(0,1)\times S, \ \  Y^1=Y^-\cup Y^+.$$
 Note that $Y^1_\ep$ is contained in its limit $Y^1$.

 The two bases of the beam $\Gamma_\ep^-$ and $\Gamma_\ep^+$ are transformed to $\Lambda^-$ and $\Lambda^+$, respectively, where
 $$\Lambda^-=\{-1\}\times S \ \mbox{ and } \ \Lambda^+=\{1\}\times S.$$
 $\Gamma_\ep$ transforms to $\Lambda=\Lambda^- \cup \Lambda^+$.

 Let $U_\epsilon \in M_\ep$ be the solution of the variational inequality (\ref{nagyegyen}). Define $u_\epsilon \in K_\ep$ by
 \begin{equation}\label{kisu}
  u_\ep(y) = U_\ep (y^{-1}_\ep(y)) \ \ \mbox{ a.e. } y\in Y_\ep.
  \end{equation}
$K_\ep$ being the image of $M_\ep$. $K_\ep$  is a closed, convex, nonempty cone in $\mathcal{D}_\ep$, with $\mathcal{D}_\ep = \{v\in H^{1}(Y_\ep)\left|\right. v=0 \mbox{ on }  \Lambda\}$. We need the following two assumptions:

{\vskip 0.3cm}

{\bf (A8)} There exists a nonempty, convex cone  $K$ in $H^1(Y^1)$ such that

{\hskip 1.3cm}(i) $K\cap H^1((-1,0)\cup (0,1)) \neq \emptyset$;

{\hskip 1.2cm}(ii) $\ep_i\to 0$, $u_{\ep_i}\in K_{\ep_i}$, $u\in H^1((-1,0)\cup (0,1))$, $u_{\ep_i}\rightharpoonup u$ (weakly) in

{\hskip 1.7cm} $H^1(Y^1)$ imply $u\in K$.


{\vskip 0.3cm}

{\bf (A9)} There exists a nonempty, convex cone $L$ in $L^2((-1,1);H^1(S))$ such

{\hskip 1cm} that

{\hskip 1.2cm} $\ep_i\to 0$, $w_{\ep_i}\in K_{\ep_i}$, $w\in L^2((-1,1);H^1(S))$, $w_{\ep_i}\rightharpoonup w$ (weakly) in

{\hskip 1.2cm}  $L^2((-1,1);H^1(S))$ imply $w\in L$.

 {\vskip 0.3cm}

By change of variables $y=y_\ep(x)$ the operator $\nabla$ transforms to
\begin{equation}\label{nablaep}\nonumber
\nabla^\ep \cdot = \left(\frac{\partial \cdot}{\partial y_1}, \frac{1}{\ep}\frac{\partial \cdot}{\partial y_2}, \frac{1}{\ep}\frac{\partial \cdot}{\partial y_3} \right).
\end{equation}






In the following we recall some results from \cite{marchis,casado2}.

\begin{lem}[\cite{marchis}]\label{corfelt3.1}
 Let $U_\ep\in M_\ep$ be the solution of the inequality (\ref{nagyegyen}) and $u_\ep \in K_\ep$ given by (\ref{kisu}). If assumptions (A1) - (A6) are verified then the sequence $U_\ep$ satisfies
 \begin{equation}\label{felt3.1}
 U_\ep \in M_\ep, \ \ \frac{1}{|\Omega_\ep|}\int_{\Omega_\ep}|\nabla U_\ep|^2 dx \leq C.
 \end{equation}
 \end{lem}

\begin{theo}[\cite{marchis}]\label{yhatarertek}
Let $U_\epsilon$ be the solution of the variational inequality (\ref{nagyegyen}) and  $u_\epsilon \in K_\ep$ defined  by $$u_\ep(y) = U_\ep (y^{-1}_\ep(y)) \ \ \mbox{ a.e. } y\in Y_\ep.$$ If assumptions (A1)-(A6) and (A8)-(A9) are verified, then
there exist three functions $u$, $w$, and $\sigma^1$ with
$$ u\in H^{1}((-1,0)\cup (0,1))\cap K, \ \ u(-1)=u(1)=0,$$
$$w\in L, \ \ \ \sigma^1 \in L^2(Y^1)^3,$$
such that up to extraction of a subsequence
 \begin{equation}\label{uhatary}\nonumber
 \chi_{Y_\ep^1}u_\ep \to u \ \ \mbox{\rm in } \ \ L^2(Y^1);
 \end{equation}
 $$\chi_{Y_\ep^-}\frac{\partial u_\ep}{\partial y_1} \rightharpoonup \frac{\partial u}{\partial y_1} \ \ \mbox{\rm in } \ \ L^2(Y^-);$$
 $$\chi_{Y_\ep^+}\frac{\partial u_\ep}{\partial y_1} \rightharpoonup \frac{\partial u}{\partial y_1} \ \ \mbox{\rm in } \ \ L^2(Y^+);$$
 $$\chi_{Y_\ep^1}\frac{1}{\ep}\nabla_{y'}u_\ep \rightharpoonup \nabla_{y'}w \ \ \mbox{\rm in } \ L^2(Y^1)^2;$$
 and
 $$\chi_{Y_\ep^1}\sigma_\ep \rightharpoonup \sigma^1 \ \ \mbox{ in } \ L^2(Y^1)^3.$$
\end{theo}

\begin{theo}[\cite{marchis}]\label{uephatar}
Let $U_\epsilon$ be the solution of the variational inequality (\ref{nagyegyen}) and $ u\in H^{1}((-1,0)\cup (0,1))\cap K$ given in Theorem \ref{yhatarertek}. If assumptions (A1)-(A6) and (A8) are verified, then there exists a subsequence of solutions $U_\ep$, also denoted by $U_\ep$, such that
\begin{equation}\label{foeq}
\lim_{\ep \to 0}\frac{1}{|\Omega_\ep|}\int_{\Omega_\ep}|U_\ep(x)-u(x_1)|^2 \d x =0.
\end{equation}
\end{theo}

 Unfortunately, this change of variables doesn't provide information about what happening near the notch. Thus we use another change of variables, which was given in \cite{casado2}. Consider the case, when
 $$\mu<+\infty \ \ \mbox{ and } \ \ \nu<+\infty.$$ The change of variables $z=z_\ep(x)$ is defined as follows
\begin{equation}\label{ch2}
z_1= \left\{\begin{array}{llll}
\left\{\begin{array}{lll}
\frac{1}{\ep r_\ep}(x_1+t_\ep)- \frac{t_\ep}{r_\ep^2}, & \mbox{if} & -1\leq x_1 \leq -t_\ep,\\
\frac{x_1}{r_\ep^2},                                   & \mbox{if} & -t_\ep\leq x_1 \leq t_\ep,    \ \ \ \ \  \mbox{ if } \mu=0,\\
\frac{1}{\ep r_\ep}(x_1-t_\ep)+ \frac{t_\ep}{r_\ep^2}, & \mbox{if} & t_\ep\leq x_1 \leq 1,
\end{array}\right.\\
\left\{\begin{array}{lll}
\frac{\mu r_\ep}{\ep t_\ep}(x_1+t_\ep)-\mu,             & \ \mbox{if} & -1\leq x_1 \leq -t_\ep,\\
\frac{\mu}{t_\ep}x_1,                                   & \ \mbox{if} & -t_\ep\leq x_1 \leq t_\ep,\    \ \ \ \ \mbox{ if } \mu>0,\\
\frac{\mu r_\ep}{\ep t_\ep}(x_1-t_\ep)+\mu,             & \ \mbox{if} & t_\ep\leq x_1 \leq 1
\end{array}\right.
\end{array}\right.
\  z'=\frac{x'}{\ep r_\ep}.
\end{equation}
This change of variables transforms the notch in a cylinder of fixed diameter and length, but transforms the rest of the beam in a very large domain.  But it allows to describe the behavior of the solution $U_\ep$ of inequality (\ref{nagyegyen}) when $x_1$ is close to zero.

We denote by $Z_\ep^-$, $Z_\ep^0$, $Z_\ep^+$, $Z_\ep$, and $Z_\ep^1$ the images of $\Omega_\ep^-$, $\Omega_\ep^0$, $\Omega_\ep^+$, $\Omega_\ep$, and $\Omega_\ep^1$ by the change of variables $z=z_\ep(x)$, i.e.
 $$Z_\ep^-=\left(-\frac{1-t_\ep}{\ep r_\ep}-\frac{t_\ep}{r_\ep^2},-\frac{t_\ep}{r_\ep^2}\right)\times \left(\frac{1}{r_\ep}S\right), \ \ Z_\ep^0=\left[-\frac{t_\ep}{r_\ep^2}, \frac{t_\ep}{r_\ep^2}\right]\times S, $$
 $$\mbox{and } \ \ Z_\ep^+=\left(\frac{t_\ep}{r_\ep^2},\frac{1-t_\ep}{\ep r_\ep}+\frac{t_\ep}{r_\ep^2}\right)\times \left(\frac{1}{r_\ep}S\right)$$
 if $\mu=0$, and
 $$Z^-_\ep=\left(-\frac{\mu r_\ep (1-t_\ep)}{\ep t_\ep}-\mu, -\mu\right)\times \left(\frac{1}{r_\ep}S \right),\ \  Z_\ep^0=[-\mu, \mu]\times S,$$
 $$\mbox{and} \ Z^+_\ep=\left(\mu, \frac{\mu r_\ep (1-t_\ep)}{\ep t_\ep}+\mu\right)\times \left(\frac{1}{r_\ep}S \right) $$
 if $\mu>0$. We set
 $$Z_\ep=Z_\ep^- \cup Z_\ep^0 \cup Z_\ep^+, \ \   Z^1_\ep = Z^-_\ep \cup Z_\ep^+.$$
We denote by $Z^-$, $Z^+$, and $Z^0$ the "limits" of $Z^-_\ep$, $Z^+_\ep$, and $Z^0_\ep$, i.e.
$$Z^-=(-\infty, -\mu)\times \R^2, \ Z^+=(\mu, +\infty)\times \R^2, \ Z^0 = [-\mu, \mu]\times S,$$
and define
$$Z=Z^-\cup Z^0\cup Z^+, \ Z^1=Z^- \cup Z^+.$$

\begin{rem}[\cite{casado2}]
In (\ref{ch2}) there are two definitions of $z_\ep$ corresponding to the cases $\mu=0$ and $\mu>0$. Actually when $\mu>0$, we could define $z_\ep$ by the definition given for $\mu=0$ because
$$\mu \sim \frac{t_\ep}{r^2_\ep},\ \ \frac{\mu r_\ep}{\ep t_\ep} \sim \frac{1}{\ep r_\ep}, \ \ \mbox{ and } \ \ \frac{\mu}{t_\ep} \sim \frac{1}{r_\ep^2}.$$
The definition (\ref{ch2}) which distinguishes the cases $\mu=0$ and $\mu>0$ has the advantage that the image $Z_\ep$ of $\Omega_\ep$ by the change of variables $z=z_\ep(x)$ is contained in its ''limit'' $Z$ for every $\ep>0$ and $Z^0_\ep$ is fixed for $\mu>0$; then a function defined in $Z$ has a restriction to $Z_\ep$.
\end{rem}


\begin{theo}[\cite{casado2}]\label{zhatarertek}
 Let $(U_\ep)_\ep$ be a sequence which satisfies (\ref{felt3.1}). Define $\hat{u}_\ep \in H^1(Z_\ep)$ by
\begin{equation}\label{kalapuep}
\hat{u}_\ep(z)=U_\ep (z_\ep^{-1}(z)), \ \ \mbox{ a.e. } z\in Z_\ep.
\end{equation}
Then there exists a  function $\hat{u}$, with
\begin{equation}\nonumber
\hat{u}\in H^1_{{\rm loc}}(Z),\ \hat{u}-u(0^-)\in L^6(Z^-), \ \hat{u}-u(0^+)\in L^6(Z^+), \ \nabla \hat{u} \in L^2(Z)^3,
\end{equation}
(where $u$ is defined in Corollary \ref{yhatarertek}), such that for every $R>0$, up to extraction of a subsequence,
$$\chi_{Z_\ep \cap B_3(0,R)}\hat{u}_\ep \to \chi _{B_3(0,R)}\hat{u} \ \ \mbox{ in } L^2(Z) \mbox{ strongly,}$$
$$\chi_{Z_\ep}\nabla \hat{u}_\ep \rightharpoonup \nabla \hat{u} \ \ \mbox{ in } L^2(Z)^3 \mbox{ weakly,}$$
where $B_3(0,R)$ denotes the 3-dimensional ball with center (0, 0, 0) and diameter $R$.
Moreover, if $\mu =0$, then $\hat{u}$ only depends on $z_1$ and satisfies
$$\hat{u}=u(0^-) \ \mbox{ in } Z^-, \ \ \hat{u}=u(0^+) \ \mbox{ in } Z^+.$$
If $\nu=\mu=0$, then $u(0^-)=u(0^+)$.\\
If $\nu =0$ and $\mu>0$, then there exists a function $\hat{w}\in L^2((-\mu, \mu);H^1(S))$ such that up to extraction of a subsequence,
\begin{equation}\nonumber\frac{r_\ep}{\ep}\nabla_{z'}\hat{u}_\ep \rightharpoonup \nabla_{z'}\hat{w} \ \ \mbox{ in } L^2(Z^0)^2 \mbox{ weakly.}\end{equation}
\end{theo}

Let $\hat{K}_\ep$ be the image of $M_\ep$ by the change of variables $z=z_\ep(x)$. $\hat{K}_\ep$  is a closed, convex, nonempty cone in $H^1(Z_\ep)$.
We need the following two assumptions:

{\vskip 0.3cm}

{\bf (A10)} There exists a nonempty subset $\hat{K}$ of $H^1_{{\rm loc}}(Z)$ such that

{\hskip 1.2cm} $\ep_i\to 0$, $R>0$, $\hat{u}_{\ep_i}\in \hat{K}_{\ep_i}$, $\hat{u}\in H^1_{{\rm loc}}(Z)$, $$\chi_{Z_\ep \cap B_3(0,R)}\hat{u}_{\ep_i}\to \chi_{B_3(0,R)}\hat{u} \ \mbox{ (strongly) in } L^2(Z),$$

and {\hskip 1.2cm} $$\chi_{Z_\ep}\nabla\hat{u}_{\ep_i}\rightharpoonup \nabla \hat{u} \ \mbox{ (weakly) in } (L^2(Z))^3,$$

imply $\hat{u}\in \hat{K}$.


{\vskip 0.3cm}

{\bf (A11)} There exists a nonempty, convex cone $\hat{L}$ in $L^2((-\mu,\mu);H^1(S))$ such that

{\hskip 1.2cm} $\ep_i\to 0$, $\hat{w}_{\ep_i}\in K_{\ep_i}$, $\hat{w}\in L^2((-\mu,\mu);H^1(S))$, $\hat{w}_{\ep_i}\rightharpoonup \hat{w}$ (weakly) in

{\hskip 1.2cm}  $L^2((-\mu,\mu);H^1(S))$ imply $\hat{w}\in \hat{L}$.

{\vskip 0.3cm}

\begin{theo}\label{enyemz}
 Let $U_\ep\in M_\ep$ be the solution of the variational inequality (\ref{nagyegyen}), $u\in H^1((-1,0)\cup(0,1))\cap K$ defined in Theorem \ref{yhatarertek}, and $\hat{u}_\ep\in \hat{K}_\ep$ given by (\ref{kalapuep}).
If assumptions (A1)-(A6) and (A8)-(A11) are verified, then there exists a  function $\hat{u}\in \hat{K}$, with
\begin{equation}
\hat{u}-u(0^-)\in L^6(Z^-), \ \hat{u}-u(0^+)\in L^6(Z^+), \ \nabla \hat{u} \in L^2(Z)^3,
\end{equation}
 such that for every $R>0$, up to extraction of a subsequence,
$$\chi_{Z_\ep \cap B_3(0,R)}\hat{u}_\ep \to \chi _{B_3(0,R)}\hat{u} \ \ \mbox{ in } L^2(Z) \mbox{ strongly,}$$
$$\chi_{Z_\ep}\nabla \hat{u}_\ep \rightharpoonup \nabla \hat{u} \ \ \mbox{ in } L^2(Z)^3 \mbox{ weakly.}$$
Moreover, if $\mu =0$, then $\hat{u}$ only depends on $z_1$ and satisfies
$$\hat{u}=u(0^-) \ \mbox{ in } Z^-, \ \ \hat{u}=u(0^+) \ \mbox{ in } Z^+.$$
If $\nu=\mu=0$, then $u(0^-)=u(0^+)$.\\
If $\nu =0$ and $\mu>0$, then there exists a function $\hat{w}\in \hat{L}$ such that up to extraction of a subsequence,
\begin{equation}\label{whatar}\frac{r_\ep}{\ep}\nabla_{z'}\hat{u}_\ep \rightharpoonup \nabla_{z'}\hat{w} \ \ \mbox{ in } L^2(Z^0)^2 \mbox{ weakly.}\end{equation}
\end{theo}

\begin{proof}
From Lemma \ref{corfelt3.1} it follows that there exists a subsequence of solutions $U_\ep$, also denoted by $U_\ep$, such that (\ref{felt3.1}) is satisfied. Thus by Theorem \ref{zhatarertek} we get that there exists a function $\hat{u} \in H^1_{{\rm loc}}(Z)$ such that the statement of the theorem is true. By assumption (A10) we get that $\hat{u}\in \hat{K}$.

If $\nu =0$ and $\mu>0$ then,  by Theorem \ref{zhatarertek}, there exists a function $\hat{w}\in L^2((-\mu, \mu);H^1(S))$ such that up to extraction of a subsequence, (\ref{whatar}) holds. Then by assumption (A11) we get that $\hat{w}\in \hat{L}$.
\end{proof}


\begin{lem}\label{lsigma0}
Let $U_\ep$ be one solution of the variational inequality (\ref{nagyegyen}), $\hat{u}_\ep$ defined by (\ref{ch2}). Assume that (A1)-(A3) and  (A5) hold. Then
$$\left\|A^0\left(\frac{\cdot}{\mu},\cdot\right)\Phi_\ep^0\left(\frac{\cdot}{\mu},\cdot,\hat{u}_\ep(\cdot)\right)B^0\left(\frac{\cdot}{\mu},\cdot\right)\nabla \hat{u}_\ep(\cdot) \right\|_{L^2(Z^0)} $$ is bounded.
\end{lem}

\begin{proof}
Taking the square of the first growth condition from (A5), multiplying by $\frac{1}{\ep^2}$, and integrating on $\Omega^0_\ep$, we obtain
\begin{align*}
&\frac{1}{\ep^2}\int_{\Omega^0_\ep}\|A_\ep(x) \Phi(x, U_\ep(x))B_\ep(x)\nabla U_\ep(x) \|^2\d x \leq \\
&\leq \frac{1}{\ep^2}\int_{\Omega^0_\ep}\| \nabla U_\ep (x)\|^2 \d x + \frac{1}{\ep^2}\int_{\Omega^0_\ep}|  U_\ep (x)|^2 \d x + \frac{|\Omega^0_\ep|}{\ep^2} \|\alpha\|_\infty.
\end{align*}
Applying the change of variable $z_\ep$ and taking out $\frac{1}{r^2_\ep}$ from $\hat{\nabla}^\ep \hat{u}_\ep$, we get
\begin{align*}
&\int_{Z^0}\left\|A^0\left(\frac{z_1}{\mu},z'\right)\Phi_\ep^0\left(\frac{z_1}{\mu},z',\hat{u}_\ep(z)\right)B^0\left(\frac{z_1}{\mu},z'\right)\nabla \hat{u}_\ep(z) \right\|^2\d z \leq \\
&\leq  C \int_{Z^0}\left\|\left(\frac{\partial \hat{u}_\ep(z)}{\partial z_1},\frac{r_\ep}{\ep}\frac{\partial \hat{u}_\ep(z)}{\partial z_2},\frac{r_\ep}{\ep}\frac{\partial \hat{u}_\ep(z)}{\partial z_3} \right) \right\|^2 \d z + r^4_\ep C \int_{Z^0}|\hat{u}_\ep(z)|^2 \d z +\bar\alpha.
\end{align*}
By Theorem \ref{zhatarertek}, $\|\nabla \hat{u}_\ep\|_{L^2(Z^0)^3}$ and $\|\hat{u}_\ep\|_{L^2(Z^0)}$ are bounded, thus the statement of the lemma holds.
\end{proof}

\begin{cor}\label{sigma0}
Suppose that the assumptions of Lemma \ref{lsigma0} are verified. Then there exists $\sigma^0 \in L^2(Z^0)$ such that
$$A^0\left(\frac{\cdot}{\mu},\cdot\right)\Phi_\ep^0\left(\frac{\cdot}{\mu},\cdot,\hat{u}_\ep(\cdot)\right)B^0\left(\frac{\cdot}{\mu},\cdot\right)\nabla \hat{u}_\ep(\cdot)  \rightharpoonup \sigma^0 \ \mbox{ in } L^2(Z^0).$$
\end{cor}


\section{The limit variational inequality}



In this section we obtain the limit problem in two cases: when $0<\mu<+\infty$ and $\nu=0$ respectively when $\mu = +\infty$ and $0<\nu<+\infty$. In these cases
$$\frac{\ep r_\ep}{t_\ep} = \frac{\ep}{r_\ep} \cdot \frac{r^2_\ep}{t_\ep} \to \frac{\nu}{\mu} = 0,$$
thus the beam has a thin neck.

{\vskip 0.3cm}

\subsection{The case $0<\mu<\infty$ and $\nu=0$}

\begin{theo}\label{limit2} Let $0<\mu<\infty$ and $\nu=0$.

Assume that  (A1)-(A11) are verified and the following four conditions are satisfied:

{\vskip 0.2cm}

{\bf (C1)}  $\varphi \in K$ implies $\chi_{Y_\ep^1}\varphi \in K_\ep$;

{\vskip 0.2cm}

{\bf (C2)}  $\psi \in L$ implies $\chi_{Y_\ep^1}\psi \in K_\ep$;

{\vskip 0.2cm}

{\bf (C3)}  $\hat\varphi \in \hat{K}$ implies $\chi_{Z_\ep^0}\hat\varphi \in \hat{K}_\ep$;

{\vskip 0.2cm}

 {\bf (C4)}  $\hat\psi \in \hat{L}$ implies $\chi_{Z_\ep^0}\hat\psi \in \hat{K}_\ep$.

 {\vskip 0.2cm}

Then the following three statements hold:

{\vskip 0.2cm}

1) There exists a subsequence of the sequence $U_\ep$ of solutions of (\ref{nagyegyen}), also denoted by $U_\ep$, and a function $u\in H^1((-1,0)\cup (0,1))\cap K$ such that (\ref{foeq}) is satisfied.

{\vskip 0.2cm}

2) Let  $u$ and $w$ be as given in Theorem \ref{yhatarertek} and $\hat{u}$ and $\hat{w}$ as in Theorem \ref{enyemz}. Then
$(u, w, \hat{u}, \hat{w})$ solves the limit variational problem:\\
find $u\in H^1((-1,0)\cup (0,1))\cap K$, $u(-1)=u(1)=0$, $w\in L$,  and $\hat{u}\in \hat{K}$, $\hat{u}(-\mu)=u(0^-)$,  $\hat{u}(\mu)=u(0^+)$, $\hat{w}\in \hat{L}$ such that for all
$v\in H^1((-1,0)\cup (0,1))\cap K$, $v(-1)=v(1)=0$, $h\in L$,  and $\hat{v}\in \hat{K}$, $\hat{v}(-\mu)=v(0^-)$,  $\hat{v}(\mu)=v(0^+)$,  $\hat{h}\in \hat{L}$,

\begin{align}\label{limithengeres0}
&\int_{Y^1}[A^1(y) \Phi^1(y,u(y_1))B^1(y)\nabla'(u, w)(y),\nabla'(v, h)(y)-\nabla'(u, w)(y)]\\
 &+ \int_{Z^0}\left[A^0\left(\frac{z_1}{\mu},z'\right) \Phi^0\left(\frac{z_1}{\mu},z',\hat{u}(z)\right)B^0\left(\frac{z_1}{\mu},z'\right)\nabla'(\hat{u},\hat{w})(z),\right.\nonumber\\
 &\left. \nabla'(\hat{v},\hat{h})(z)-\nabla'(\hat{u},\hat{w})(z)
  \right] \d z \geq 0.\nonumber
\end{align}

{\vskip 0.2cm}

3) Let $\sigma^1$ be as given in Theorem \ref{yhatarertek}, $\sigma^0$ as given in Corollary \ref{sigma0}. Then $$\sigma^1 (y) = A^1(y) \Phi^1(y,u(y))B^1(y)\nabla'(u,w)(y) \ \ \mbox{ for a.e. } y\in Y^1,$$
$$\sigma^0(z) = A^0\left(\frac{z_1}{\mu},z'\right) \Phi^0\left(\frac{z_1}{\mu},z',\hat{u}(z)\right) B^0\left(\frac{z_1}{\mu},z'\right) \nabla' \left(\hat{u},\frac{1}{\nu}\hat{u}\right)$$
$  \ \ \mbox{ for a.e. } z\in Z^0.$

\end{theo}

\begin{proof} Statement 1) follows from Theorem \ref{uephatar}.

2)
Since $\nu=0$, from Theorem \ref{enyemz} it follows that $\hat{u}\in \hat{K}$ only depends on $z_1$  with
$$\hat{u}=u(0^-) \ \mbox{ in } Z^-, \ \ \hat{u}=u(0^+) \ \mbox{ in } Z^+,$$
and there exists a function $\hat{w}\in \hat{L}$ such that up to extraction of a subsequence,
$$\frac{r_\ep}{\ep}\nabla_{z'}\hat{u}_\ep \rightharpoonup \nabla_{z'}\hat{w} \ \ \mbox{ in } L^2(Z^0)^2 \mbox{ weakly.}$$

Let $\varphi^-\in H^1([-1,0])$ and $\varphi^+\in H^1([0,1])$
and define $\varphi \in H^1((-1,0)\cup(0,1))\cap K$ such that
\begin{equation}\nonumber
\varphi(x_1)=\begin{cases}
\varphi^-(x_1), \ \ \mbox{ if } \ x_1\in (-1,0)\\
\varphi^+(x_1), \ \ \mbox{ if } \ x_1\in (0,1).
\end{cases}
\end{equation}
Let  $\psi\in L$,
 $\hat{\varphi}\in \hat{K}$, and
 $\hat{\psi}\in \hat{L}$.
For $\ep$ small enough, the sequence $V_\ep$ defined by
\begin{align*}
V_\ep(x) & = \chi_{\Omega_\ep^1}(x)\left(\varphi(x_1) + \ep \psi \left(x_1, \frac{x'}{\ep} \right) \right)+\\
&+ \chi_{\Omega_\ep^0}(x)\left(\hat\varphi\left(\frac{\mu x_1}{t_\ep}\right) + \frac{\ep}{r_\ep} \hat\psi \left(\frac{\mu x_1}{t_\ep}, \frac{x'}{\ep r_\ep} \right) \right), \ \ \mbox{ a.e. } \ x\in \Omega_\ep
\end{align*}
belongs to $M_\ep$.

Putting $\eta=U_\ep(x)$, $\xi=\nabla U_\ep(x)$ and
\begin{align*}
\tau=\tau_\ep(x)&=\chi_{\Omega_\ep^1}(x) (\nabla'(\varphi, \psi)+\lambda f_1)(y_\ep(x))+\\
&+\chi_{\Omega_\ep^0}(x)\frac{1}{r^2_\ep}(\nabla'(\hat{\varphi},\hat{\psi})+\lambda f_2)(z_\ep(x)), \ \mbox{ a.e. } x\in \Omega_\ep
\end{align*}
 in the monotonicity condition, we get
\begin{align*}
0&\leq \frac{1}{\ep^2} \int_{\Omega_\ep}\left[A_\ep(x) \Phi_\ep(x, U_\ep(x)) B_\ep(x) \nabla U_\ep(x)-A_\ep(x) \Phi_\ep(x, U_\ep(x)) B_\ep(x) \tau_\ep(x), \right.\\
&\ \ \ \ \ \ \ \ \ \ \ \ \ \left. \nabla U_\ep(x)-\tau_\ep(x) \right] \d x =\\
&= \frac{1}{\ep^2} \int_{\Omega_\ep}\left[A_\ep(x) \Phi_\ep(x, U_\ep(x)) B_\ep(x) \nabla U_\ep(x), \nabla U_\ep(x)\right]\d x - \\
&- \frac{1}{\ep^2} \int_{\Omega_\ep}\left[A_\ep(x) \Phi_\ep(x, U_\ep(x)) B_\ep(x) \nabla U_\ep(x), \tau_\ep(x)\right]\d x +\\
&- \frac{1}{\ep^2} \int_{\Omega_\ep}\left[A_\ep(x) \Phi_\ep(x, U_\ep(x)) B_\ep(x) \tau_\ep(x), \nabla U_\ep(x)\right]\d x - \\
&+\frac{1}{\ep^2} \int_{\Omega_\ep}\left[A_\ep(x) \Phi_\ep(x, U_\ep(x)) B_\ep(x) \tau_\ep(x), \tau_\ep(x)\right]\d x =\\
&= T_1^\ep - T_2^\ep - T_3^\ep + T_4^\ep.
\end{align*}
In the following we study each term separately. The first term
\begin{eqnarray}
T_1^\ep & = &\frac{1}{\ep^2} \int_{\Omega_\epsilon}\left[A_\epsilon(x) \Phi_\epsilon(x,U_\epsilon(x))B_\epsilon(x)\nabla U_\epsilon(x),\nabla U_\epsilon(x) \right] \d x\leq\nonumber\\
&\leq & \frac{1}{\ep^2}\int_{\Omega_\epsilon}\left[A_\epsilon(x) \Phi_\epsilon(x,U_\epsilon(x))B_\epsilon(x)\nabla U_\epsilon(x),\nabla V_\epsilon(x)  \right] \d x \nonumber\\
&=&\frac{1}{\ep^2}\int_{\Omega_\epsilon^1}\left[A^1_\epsilon(y_\ep(x)) \Phi_\epsilon^1(y_\ep(x),U_\epsilon(x))B_\epsilon^1(y_\ep(x))\nabla U_\epsilon(x),\right.\nonumber\\
&&\left.  \left(\frac{\d\varphi(x_1)}{\d x_1}+\ep \frac{\partial \psi(y_\ep(x))}{\partial x_1}, \frac{\partial \psi(y_\ep(x))}{\partial x_2}, \frac{\partial \psi(y_\ep(x))}{\partial x_3} \right)  \right] \d x  +\nonumber\\
&&+\frac{1}{\ep^2}\int_{\Omega_\epsilon^0}\left[A^0_\epsilon(z_\ep(x)) \Phi_\epsilon^0(z_\ep(x),U_\epsilon(x))B_\epsilon^0(z_\ep(x))\nabla U_\epsilon(x),\right.\nonumber\\
&&\left.  \left(\frac{\mu}{t_\ep}\frac{\partial \hat\varphi\left(\frac{\mu x_1}{t_\ep} \right)}{\partial x_1} + \frac{\ep \mu}{r_\ep t_\ep} \frac{\partial \hat\psi (z_\ep(x))}{\partial x_1}, \frac{1}{r^2_\ep}\frac{\partial\hat\psi (z_\ep(x))}{\partial x_2}, \frac{1}{r^2_\ep}\frac{\partial \hat\psi (z_\ep(x))}{\partial x_3}\right)  \right] \d x\nonumber
\end{eqnarray}
(using the change of variable $y=y_\ep(x)$ in the integral over $\Omega^1_\ep$ and the change of variables $z=z_\ep(x)$ in the integral over $\Omega^0_\ep$)
\begin{align*}
& = \int_{Y_\ep^1}\left[A^1(y) \Phi_\epsilon^1(y,u_\ep(y))B^1(y)\nabla^\ep u_\ep(y),  \right.\\
&\ \ \ \ \ \ \ \ \left.\left(\frac{\d\varphi(y_1)}{\d y_1}+\ep \frac{\partial \psi(y)}{\partial y_1}, \frac{\partial \psi(y)}{\partial y_2}, \frac{\partial \psi(y)}{\partial y_3} \right)
  \right] \d y + \\
  &+ \frac{1}{\mu}t_\ep r_\ep^2\int_{Z^0}\left[A^0\left(\frac{z_1}{\mu},z'\right) \Phi_\ep^0\left(\frac{z_1}{\mu},z',\hat{u}(z)\right)B^0 \left(\frac{z_1}{\mu},z'\right)\cdot \right.\\
  &  \ \ \ \ \ \ \ \ \ \ \ \ \ \ \ \ \cdot \left. \left(\frac{\mu}{t_\ep}\frac{\partial \hat{u}_\ep (z)}{\partial z_1}, \frac{1}{\ep r_\ep}\frac{\partial \hat{u}_\ep (z)}{\partial z_2}, \frac{1}{\ep r_\ep}\frac{\partial \hat{u}_\ep (z)}{\partial z_3}\right)
  \right.,\\ &\left. \ \ \ \ \ \ \ \ \ \ \ \ \ \ \ \ \
  \left(\frac{\mu}{t_\ep}\frac{\d \hat{\varphi} (z_1)}{\d z_1}+\frac{\ep}{r_\ep t_\ep}\frac{\partial \hat{\psi} (z)}{\partial z_1}, \frac{1}{r_\ep^2}\frac{\partial \hat{\psi} (z)}{\partial z_2}, \frac{1}{r_\ep^2}\frac{\partial \hat{\psi} (z)}{\partial z_3}\right) \right] \d z
\end{align*}
Taking the limit, we get
\begin{align*}
T_1^\ep &\to \int_{Y^1}\left[\sigma^1(y), \nabla'(\varphi, \psi)(y)\right] \d y
+ \int_{Z^0}\left[\sigma^0(z),\nabla'(\hat{\varphi},\hat{\psi})(z) \right] \d z.
\end{align*}

The second term
\begin{align*}
T_2^\ep& =  \frac{1}{\ep^2} \int_{\Omega_\ep}\left[A_\ep(x) \Phi_\ep(x, U_\ep(x)) B_\ep(x) \nabla U_\ep(x), \tau_\ep(x)\right]\d x \to\\
& \to \int_{Y^1}\left[\sigma^1(y), (\nabla'(\varphi, \psi)+\lambda f_1)(y)\right] \d y +\\ &+\int_{Z^0}\left[\sigma^0(z),(\nabla'(\hat{\varphi},\hat{\psi})+\lambda f_2)(z) \right] \d z,
\end{align*}
when $\ep$ tends to zero.

The third term
\begin{align*}
T_3^\ep& =  \frac{1}{\ep^2} \int_{\Omega_\ep}\left[A_\ep(x) \Phi_\ep(x, U_\ep(x)) B_\ep(x) \tau_\ep(x), \nabla U_\ep(x)\right]\d x \to \\
& \to \int_{Y^1}\left[A^1(y) \Phi^1(y, u(y)) B^1(y) (\nabla'(\varphi,\psi)+\lambda f_1)(y), \nabla'(u,w)(y)\right]\d y+\\
&+\int_{Z^0}\left[A^0\left(\frac{z_1}{\mu},z'\right) \Phi^0\left(\frac{z_1}{\mu},z',\hat{u}(z)\right)B^0\left(\frac{z_1}{\mu},z'\right)(\nabla'(\hat{\varphi},\hat{\psi})+\lambda f_2)(z),\right.\\
&\left. \ \ \ \ \ \ \ \ \ \nabla'(\hat{u},\hat{w})(z),
  \right] \d z,
\end{align*}
when $\ep$ tends to zero.

The last term
\begin{align*}
T_4^\ep &=\frac{1}{\ep^2} \int_{\Omega_\ep}\left[A_\ep(x) \Phi_\ep(x, U_\ep(x)) B_\ep(x) \tau_\ep(x), \tau_\ep(x)\right]\d x \to \\
&\to\int_{Y^1}\left[A^1(y) \Phi^1(y, u(y)) B^1(y) (\nabla'(\varphi,\psi)+\lambda f_1)(y), \right.\\
&\ \ \ \ \ \ \ \ \left.(\nabla'(\varphi,\psi)+\lambda f_1)(y)\right]\d y +\\
&+\int_{Z^0}\left[A^0\left(\frac{z_1}{\mu},z'\right) \Phi^0\left(\frac{z_1}{\mu},z',\hat{u}(z)\right)B^0\left(\frac{z_1}{\mu},z'\right)(\nabla'(\hat{\varphi},\hat{\psi})+\lambda f_2)(z),\right.\\
&\left. \ \ \ \ \ \ \ \ \ (\nabla'(\hat{\varphi},\hat{\psi})+\lambda f_2)(z)
  \right] \d z,
\end{align*}
when $\ep$ tends to zero.


Adding the limits of $T_1^\ep$, $T_2^\ep$, $T_3^\ep$, and $T_4^\ep$, we get
\begin{align}\label{limitseged2}
&-\int_{Y^1}[\sigma^1(y),\lambda f_1(y)]\d y - \int_{Z^0}[\sigma^0(z),\lambda f_2(z)]\d z + \\
&+\int_{Y^1}[A^1(y) \Phi^1(y,u(y_1))B^1(y)(\nabla'(\varphi,\psi)+\lambda f_1)(y),\nabla'(\varphi, \psi)(y)-\nonumber\\
& \ \ \ \ \ \  -\nabla'(u, w)(y)+\lambda f_1(y)]
 + \nonumber
\end{align}
\begin{align}
 &+\int_{Z^0}\left[A^0\left(\frac{z_1}{\mu},z'\right) \Phi^0\left(\frac{z_1}{\mu},z',\hat{u}(z)\right)B^0\left(\frac{z_1}{\mu},z'\right)(\nabla'(\hat{\varphi},\hat{\psi})+\lambda f_2)(z),\right.\nonumber\\
&\left. \ \ \ \ \ \ \ \ \ \nabla'(\hat{\varphi},\hat{\psi})(z)-\nabla'(\hat{u},\hat{w})(z)+\lambda f_2(z),
  \right] \d z \geq 0.\nonumber
\end{align}
Setting
$$\varphi - u = \theta (v-u), \ \ \psi - w = \theta (h - w), \ \  \hat{\varphi}=\theta \hat{v}, \ \ \mbox{ and } \hat{\psi}=\theta\hat{h},$$
where $\theta>0$, dividing by $\theta$, then letting $\theta \to 0$, we get the limit variational inequality.

Putting
$$(\varphi, u)=(\psi, w) \ \ \mbox{ and } \ \ (\hat\varphi, \hat{u})=(\hat\psi, \hat{w}),$$
dividing by $\lambda$, and letting $\lambda\to 0$, we get
\begin{align*}
& \int_{Y^1} [\sigma^1(y)-A^1(y) \Phi^1(y,u(y_1))B^1(y)\nabla'(u,w)(y),f_1(y)]\d y +\\
&+ \int_{Z^0}\left[\sigma^0(z)-A^0\left(\frac{z_1}{\mu},z'\right) \Phi^0\left(\frac{z_1}{\mu},z',\hat{u}(z)\right) B^0\left(\frac{z_1}{\mu},z'\right) \nabla' \left(\hat{u},\hat{w}\right)(z),\right.\\
&\left. \ \ \ \ \ \ \ \ \ \   f_2(z)\right]\d z \geq 0, \quad  \forall f_1 \in H^1(Y^1), \forall f_2 \in H^1(Z).
\end{align*}
Then 3) follows.
\end{proof}

\subsection{The case $\mu = +\infty$ and $0<\nu <+\infty$}

\begin{theo}\label{limit3} Let $\mu = +\infty$ and $0<\nu <+\infty$.
Assume that (A1)-(A9) are verified and the following two conditions are satisfied:

{\vskip 0.2cm}

{\bf (C1)} $\varphi \in K$ implies $\chi_{Y_\ep^1}\varphi \in K_\ep$;

{\vskip 0.2cm}

{\bf (C2)} $\psi \in L$ implies $\chi_{Y_\ep^1}\psi \in K_\ep$.

{\vskip 0.2cm}

Then the following three statements hold:

{\vskip 0.2cm}

1) There exists a subsequence of the sequence $U_\ep$ of solutions of (\ref{nagyegyen}), also denoted by $U_\ep$, and a function $u\in H^1((-1,0)\cup (0,1))\cap K$ such that (\ref{foeq}) is satisfied.

{\vskip 0.2cm}

2)  Let $u$ and $w$ be given as in Theorem \ref{yhatarertek}. Then $(u, w)$ solves the limit variational problem:\\
find $u\in H^1((-1,0)\cup (0,1))\cap K$, $u(-1)=u(1)=0$ and $w\in L$  such that for all
$v\in H^1((-1,0)\cup (0,1))\cap K$, $v(-1)=v(1)=0$ and $h\in L$
\begin{align}\label{limithengeresinf}
&\int_{Y^1}[A^1(y) \Phi^1(y,u(y_1))B^1(y)\nabla'(u, w)(y),\nabla'(v, h)(y)-\nabla'(u, w)(y)]
 \geq 0.\end{align}

{\vskip 0.2cm}

3) Let $\sigma^1$ given in Theorem \ref{yhatarertek}. Then $$\sigma^1 (y) = A^1(y) \Phi^1(y,u(y))B^1(y)\nabla'(u,w)(y) \ \ \mbox{ for a.e. } y\in Y^1.$$
\end{theo}

\begin{proof} Statement 1) follows from Theorem \ref{uephatar}.


To prove statement 2),
let $\varphi^-\in H^1([-1,0])$ and $\varphi^+\in H^1([0,1])$
and define $\varphi \in H^1((-1,0)\cup(0,1))\cap K$ such that
\begin{equation}\nonumber
\varphi(x_1)=\begin{cases}
\varphi^-(x_1), \ \ \mbox{ if } \ x_1\in (-1,0)\\
\varphi^+(x_1), \ \ \mbox{ if } \ x_1\in (0,1).
\end{cases}
\end{equation}
Let  $\psi\in L$
and $\gamma^0: [0, +\infty) \to \R$ defined by
\begin{equation}\nonumber
\gamma^0(\tau)=
\begin{cases}
\tau, \ \mbox{  if  } 0\leq \tau \leq 1\\
1,  \ \mbox{  if  } \tau\geq  1.
\end{cases}
\end{equation}
and
$$V_\ep(x)=\varphi(x_1)\gamma^0\left(\frac{|x_1|}{t_\ep}\right)+\ep \psi\left(x_1,\frac{x'}{\ep}\right), \ \mbox{ a.e } \in \Omega_\ep,$$
which belongs to $M_\ep$.\\
For $\ep$ small enough, by a simple calculation we obtain
\begin{align*}
\frac{1}{\ep^2}&\int_{\Omega_\ep^1}\left|\nabla V_\ep - \frac{\d \varphi(x_1)}{\d x_1}e_1-\nabla_{y'}\psi\left(x_1, \frac{x'}{\ep}\right)\right| \d x + \frac{1}{\ep^2}\int_{\Omega_\ep^0} |\nabla V_\ep| \d x \leq\\
&\leq C \left(\ep^2+\frac{r_\ep^2}{t_\ep}\right)
\end{align*}
which tends to zero since $\mu=+\infty$.

Putting $\eta=U_\ep(x)$, $\xi=\nabla U_\ep(x)$ and
\begin{equation}\nonumber
\tau=\tau_\ep(x)=
\left\{\begin{array}{lll}
 (\nabla'(\varphi, \psi)+\lambda f_1)(y_\ep(x)), & \mbox{if} & x\in \Omega_\ep^1\\
 0, & \mbox{if} & x\in \Omega_\ep^0
 \end{array}\right.
 \end{equation}
 in the monotonicity condition, we get
\begin{align*}
0&\leq \frac{1}{\ep^2} \int_{\Omega_\ep}\left[A_\ep(x) \Phi_\ep(x, U_\ep(x)) B_\ep(x) \nabla U_\ep(x)-A_\ep(x) \Phi_\ep(x, U_\ep(x)) B_\ep(x) \tau_\ep(x), \right.\\
& \left. \ \ \ \ \ \ \ \ \ \ \ \ \nabla U_\ep(x)-\tau_\ep(x) \right] \d x =\\
&= \frac{1}{\ep^2} \int_{\Omega_\ep}\left[A_\ep(x) \Phi_\ep(x, U_\ep(x)) B_\ep(x) \nabla U_\ep(x), \nabla U_\ep(x)\right]\d x - \\
&- \frac{1}{\ep^2} \int_{\Omega_\ep}\left[A_\ep(x) \Phi_\ep(x, U_\ep(x)) B_\ep(x) \nabla U_\ep(x), \tau_\ep(x)\right]\d x -\\
&- \frac{1}{\ep^2} \int_{\Omega_\ep}\left[A_\ep(x) \Phi_\ep(x, U_\ep(x)) B_\ep(x) \tau_\ep(x), \nabla U_\ep(x)\right]\d x + \\
&+\frac{1}{\ep^2} \int_{\Omega_\ep}\left[A_\ep(x) \Phi_\ep(x, U_\ep(x)) B_\ep(x) \tau_\ep(x), \tau_\ep(x)\right]\d x =\\
&= T_1^\ep - T_2^\ep - T_3^\ep + T_4^\ep.
\end{align*}

In the following we study each term separately. The first term
\begin{align*}
T_1^\ep & =\frac{1}{\ep^2} \int_{\Omega_\epsilon}\left[A_\epsilon(x) \Phi_\epsilon(x,U_\epsilon(x))B_\epsilon(x)\nabla U_\epsilon(x),\nabla U_\epsilon(x) \right] \d x\leq\\
&\leq \frac{1}{\ep^2}\int_{\Omega_\epsilon}\left[A_\epsilon(x) \Phi_\epsilon(x,U_\epsilon(x))B_\epsilon(x)\nabla U_\epsilon(x),\nabla V_\epsilon(x)  \right] \d x =\\
&=\frac{1}{\ep^2}\int_{\Omega_\epsilon^1}\left[A_\epsilon(x) \Phi_\epsilon(x,U_\epsilon(x))B_\epsilon(x)\nabla U_\epsilon(x),\nabla V_\epsilon(x)  \right] \d x +\\
&+\frac{1}{\ep^2}\int_{\Omega_\epsilon^0}\left[A_\epsilon(x) \Phi_\epsilon(x,U_\epsilon(x))B_\epsilon(x)\nabla U_\epsilon(x),\nabla V_\epsilon(x)  \right] \d x,
\end{align*}
where the second term tends to zero. We use the change of variables $y=y_\ep(x)$ in the first term:
\begin{align*}
T_1^\ep & \leq \int_{Y_\epsilon^1}\left[A^1(y) \Phi^1_\epsilon(y,u_\epsilon(y))B^1(y)\nabla^\ep u_\epsilon(y),\right.\\
&\left. \ \ \ \ \ \ \ \ \left(\frac{\d\varphi(y_1)}{\d y_1}+\ep \frac{\partial \psi(y)}{\partial y_1}, \frac{\partial \psi(y)}{\partial y_2}, \frac{\partial \psi(y)}{\partial y_3} \right)  \right] \d y + O_\ep = \\
\end{align*}
\begin{align*}
&=\int_{Y^1}\left[A^1(y) \Phi^1_\epsilon(y,u_\epsilon(y))B^1(y)\nabla^\ep u_\epsilon(y),\right.\\
&\left. \ \ \ \ \ \ \ \ \left(\frac{\d\varphi(y_1)}{\d y_1}+\ep \frac{\partial \psi(y)}{\partial y_1}, \frac{\partial \psi(y)}{\partial y_2}, \frac{\partial \psi(y)}{\partial y_3} \right)  \right] \d y + O_\ep.
\end{align*}
Taking the limit of both sides, we get
\begin{align*}
\lim_{\ep\to 0} T_1^\ep & \leq \int_{Y^1}\left[\sigma^1(y),\nabla'(\varphi, \psi)(y)\right]\d y.
\end{align*}

\noindent The third term
\begin{align*}
T_3^\ep &= \frac{1}{\ep^2} \int_{\Omega_\ep}\left[A_\ep(x) \Phi_\ep(x, U_\ep(x)) B_\ep(x) \tau_\ep(x), \nabla U_\ep(x)\right]\d x =\\
&= \frac{1}{\ep^2} \int_{\Omega_\ep^1}\left[A^1(y_\ep(x)) \Phi_\ep(y_\ep(x), U_\ep(x)) B^1(y_\ep(x)) (\nabla'(\varphi,\psi)+\lambda f_1)(y_\ep(x)), \right.\\
&\left.\ \ \ \ \ \ \ \ \  \ \ \ \ \ \ \nabla U_\ep(x)\right]\d x,
\end{align*}
as the integral on $\Omega^0_\ep$ is equal with zero because $\tau_\ep=0$ on $\Omega^0_\ep$. Using the change of variable $y=y_\ep(x)$ we get
\begin{align*}
T_3^\ep &= \int_{Y_\ep^1}\left[A^1(y) \Phi_\ep(y, u_\ep(y)) B^1(y) (\nabla'(\varphi,\psi)+\lambda f_)(y), \nabla^\ep u_\ep(y)\right]\d y = \\
&=\int_{Y^1}\left[A^1(y) \Phi_\ep(y, u_\ep(y)) B^1(y) (\nabla'(\varphi,\psi)+\lambda f_1)(y), \nabla^\ep u_\ep(y)\right]\d y  + O_\ep.
\end{align*}
Taking the limit when $\ep \to 0$, we get
\begin{align*}
T_3^\ep
&\to\int_{Y^1}\left[A^1(y) \Phi(y, u(y_1)) B^1(y) (\nabla'(\varphi,\psi)+\lambda f_1)(y), \nabla'(u,w)(y)\right]\d y.
\end{align*}
Similarly
\begin{align*}
T_2^\ep
&\to\int_{Y^1}\left[\sigma^1(y), (\nabla'(\varphi,\psi)+\lambda f_1)(y)\right]\d y
\end{align*}
and
\begin{align*}
T_4^\ep
&\to\int_{Y^1}\left[A^1(y) \Phi(y, u(y_1)) B^1(y) (\nabla'(\varphi,\psi)+\lambda f_1)(y),\right. \\
&\ \ \ \ \ \ \ \ \ \ \  \left.(\nabla'(\varphi,\psi)+\lambda f_1)(y)\right]\d y,
\end{align*}
when $\ep \to 0$.

Adding the limits of $T_1^\ep$, $T_2^\ep$, $T_3^\ep$, and $T_4^\ep$, we get
\begin{align}\label{limitseged3}
&\int_{Y^1}[A^1(y) \Phi^1(y,u(y_1))B^1(y)(\nabla'(\varphi,\psi)+\lambda f_1)(y),\nabla'(\varphi, \psi)(y)-\\
&\ \ \ \ \ -\nabla'(u, w)(y)+\lambda f_1(y)]\d z \nonumber
- \int_{Y^1}[\sigma^1(y), \lambda f_1(y)]\d y\geq 0.\nonumber
\end{align}
Setting
$$\varphi - u = \theta (v-u), \ \ \mbox{ and } \ \  \psi - w = \theta (h - w), $$
where $\theta>0$, dividing by $\theta$, then letting $\theta \to 0$, we get the limit variational inequality.

3) Putting
$$(\varphi, u)=(\psi, w),$$
dividing by $\lambda$, and letting $\lambda\to 0$, we get
\begin{align*}
& \int_{Y^1} [\sigma^1(y)-A^1(y) \Phi^1(y,u(y_1))B^1(y)\nabla'(u,w)(y),f_1(y)]\d y  \geq 0\\
 & \ \forall f_1 \in H^1(Y^1).
\end{align*}
Then 3) follows.
\end{proof}

\subsection*{Acknowledgement} This research was partially
supported by the research grant CNCSIS PN II IDEI523/2007.



\bigskip
\rightline{\emph{Received: October 11, 2009; Revised: April 3, 2010}}     


\begin{thebibliography}{9}


\bibitem{bruno} P. Bruno, Geometrically constrained magnetic wall, {\it Phys. Rev. Lett.}, {\bf 83} (1999), 2425--2428.

\bibitem{cabib} E. Cabib, L. Freddi, A. Morassi, D. Percivale, Thin notched beams, {\it J. Elasticity}, {\bf 64} (2002), 157--178.


\bibitem{casado1} J. Casado-D\'iaz, M. Luna-Laynez, F. Murat, { Asymptotic behavior of diffusion problems in a domain made of two cylinders of different diameters and lengths}, {\it C. R. Acad. Sci. Paris, S\'er. I}, {\bf 338} (2004), 133--138.


\bibitem{casado2} J. Casado-D\'iaz, M. Luna-Laynez, F. Murat, The diffusion equation in a notched beam, {\it Calculus of Variations}, {\bf 31} (2008), 297--323.

\bibitem{mossino} P. Courilleau, J. Mossino, { Compensated
 compactness for nonlinear homogenization and reduction of dimension}, {\it Calculus of
 Variations}, {\bf 20} (2004), 65--91.


\bibitem{gustafsson}
B. Gustafsson, J. Mossino, {Non-periodic explicit homogenization and
 reduction of dimension: the linear case}, {\it IMA Journal of Applied
 Mathematics} {\bf 68} (2003), 269--298.





\bibitem{hale} J. K. Hale, J. Vegas, A nonlinear parabolic equation with varying domain, {\it Arch. Rat. Mech. Anal.}, {\bf 86} (1984), 99--123.

\bibitem{jimbo1} S. Jimbo, Singular perturbation of domains and semilinear elliptic equation, {\it J. Fac. Sci. Univ. Tokyo}, {\bf 35} (1988), 27--76.

\bibitem{jimbo2} S. Jimbo, Singular perturbation of domains and semilinear elliptic equation 2, {\it J. Diff. Equat.}, {\bf 75} (1988), 264--289.


\bibitem{kohn} R. V. Kohn, V. V. Slastikov, Geometrically constrained walls, {\it  Calculus of Variations and Partial Differential Equations}, {\bf 28}  (2007), 33--57.



\bibitem{marchis} I. Marchis, { Asymptotic behavior of the solution of nonlinear parametric variational inequalities in notched beams}, accepted in {\it Studia Universitatis Babe\c{s}-Bolyai, Mathematica}.

\bibitem{muratsili} F. Murat, A. Sili, { Probl\`emes monotones dans des cylindres de faible diam\`etre form\'es de mat\'eriaux h\'et\'erog\`enes}, {\it C. R. Acad. Sci. Paris, S\'er. I}, {\bf 320} (1995), 1199--1204.









\bibitem{rubi} J. Rubinstein, M. Schatzman, P. Sternberg, Ginzburg-Landau model in thin loops with narrow contrictions, {\it SIAM J. Appl. Math.}, {\bf 64} (2004), 2186--2204.

\bibitem{sili} A. Sili, { Asymptotic behavior of the solutions of monotone problems in flat cylinders}, {\it Asymptotic Analysis}, {\bf 19} (1999), 19--33.



\end{thebibliography}
\end{document}